\documentclass{amsart}
\usepackage{graphicx}
\newtheorem{thm}{Theorem}[section]

\newtheorem{prop}[thm]{Proposition}
\theoremstyle{definition}

\theoremstyle{remark}

\numberwithin{equation}{section}

\begin{document}
\title[ functional equations on monoids]{A class of functional equations     on monoids  }
\author[B. Bouikhalene and  E. Elqorachi]{ Bouikhalene Belaid and  Elqorachi Elhoucien }
\begin{abstract} In \cite{05} B. Ebanks and H. Stetk{\ae}r obtained the solutions of the functional equation $f(xy)-f(\sigma(y)x)=g(x)h(y)$ where $\sigma$ is an
 involutive  automorphism and $f,g,h$  are complex-valued functions, in the setting of a group $G$ and a monoid $M$.
  Our main goal is to determine the  complex-valued solutions of the following more general version of this
equation, viz $f(xy)-\mu(y)f(\sigma(y)x)=g(x)h(y)$ where $\mu:
G\longrightarrow \mathbb{C}$ is a  multiplicative function such that
$\mu(x\sigma(x))=1$ for all $x\in G$. As an application we find the
complex-valued solutions $(f,g,h)$ on groups of the equation
$f(xy)+\mu(y)g(\sigma(y)x)=h(x)h(y)$.
 \end{abstract}

\maketitle


\section{Introduction}We recall that a semigroup $S$ is a
set equipped with an associative operation. We write the operation
multiplicatively. A monoid is a semigroup $M$ with identity element
that we denote $e$.  A function $\chi$ : $S\longrightarrow
\mathbb{C}$ is said to be multiplicative if
$\chi(xy)=\chi(x)\chi(y)$ for all
$x,y\in S.$\\
Let $S$ be a semigroup and $\sigma$ : $G\longrightarrow G$ an
involutive homomorphism, that is $\sigma(xy)=\sigma(x)\sigma(y)$ and
$\sigma(\sigma(x))=x$ for all $x,y\in G.$ \\
The complex-valued solutions of the following variant of
d'Alembert's functional equation
\begin{equation}\label{eq1}
f(xy)+f(\sigma(y)x)=2f(x)f(y),\;x,y\in S.
\end{equation} was determined by Stetk\ae r \cite{007}. They are the functions  of the form
$$f(x)=\frac{\chi+\chi\circ\sigma}{2},$$ where $\chi$ : $G\longrightarrow
\mathbb{C}$ is multiplictive.\\In the same year, Ebanks and Stetk\ae
r \cite{05} obtained on monoids the complex-valued solutions of the
following d'Alembert's other functional equation
\begin{equation}\label{eq2}
    f(xy)-f(\sigma(y)x)=g(x)h(y),\; x,y\in S.
\end{equation}
This functional equation contains, among others, an equation of
d'Alembert \cite{01,02,03}
\begin{equation}\label{eq3}
f(x+y)-f(x-y)=g(x)h(y),\;x,y\in \mathbb{R}\end{equation}
 the solutions of which are known on  abelian groups, and a functional
equation
\begin{equation}\label{eq4}
    f(x+y)-f(x+\sigma(y))=g(x)h(y),\;x,y\in G
\end{equation} studied by Stetk\ae r [12, Corollary III.5] on abelian groups $G$.\\
There are various ways of extending functional equations from
abelian groups to non-abelian groups. The $\mu$-d'Alembert
functional equation
\begin{equation}\label{eq5}
f(xy)+\mu(y)f(x\tau(y))=2f(x)f(y),\; x,y\in G\end{equation} which is
an extension of d'Alembert functional equation
\begin{equation}\label{eq6}
f(xy)+f(x\tau(y))=2f(x)f(y),\; x,y\in G,\end{equation}  where $\tau$
is an involutive of  $G$, is closely related to pre-d'Alembert
functions. It occurs in the literature. See Parnami, Singh and
Vasudeva \cite{parnami}, Davison  [5, Proposition 2.11], Ebanks and
Stetk\ae r \cite{st2}, Stetk\ae r [15, Lemma IV.4] and Yang [16,
Proposition 4.2]. The functional equation (\ref{eq5}) has been
treated systematically by Stetk\ae r \cite{st3}. The non-zero
solutions of (\ref{eq5})
 are the normalized traces of certain representations of $G$ on $\mathbb{C}^{2}$. Davison proved this via his
 work \cite{davison} on the pre-d'Alembert functional equation on
 monoids.\\The variant of Wilson's functional equation
 \begin{equation}\label{eq77777}
f(xy)+\mu(y)f(\tau(y)x)=2f(x)g(y),\; x,y\in G;\end{equation}
 \begin{equation}\label{eq7}
f(xy)+\mu(y)f(\sigma(y)x)=2f(x)g(y),\; x,y\in G\end{equation} with
$\mu\neq 1$ was recently studied on groups by Elqorachi and Redouani
\cite{elqorachi}.
\\ The complex-valued solutions  of equation (\ref{eq77777}) with $\tau(x)=x^{-1}$ and
$\mu(x)=1$ for all $x\in G$ are obtained on groups by Ebanks and
Stetk\ae r \cite{st2}.\\ The present paper complements and contains
the existing results about (\ref{eq2}) by finding the solutions
$f,g,h$ of the extension \begin{equation}\label{eq8}
    f(xy)-\mu(y)f(\sigma(y)x)=g(x)h(y),\; x,y\in S
\end{equation}of it to monoids that need not be abelian, because on non-abelian monoids the order of factors matters and
 involutions and involutive automorphisms differ. \\ As in  \cite{05} one of the main
ideas is to relate the functional equation (\ref{eq2}) to a sine
substraction law on monoids. In our case we need the solutions  of
the following version of the sine subtraction law
\begin{equation}\label{eq9}\mu(y)f(x\sigma(y))=f(x)g(y)-f(y)g(x), \; x,y\in
S.
\end{equation} These results are
obtained in Theorem 2.1. We need also the solutions of equation
(\ref{eq7}) on monoids. There are not in the literature, but we
derived them in Theorem 3.2.  In   section 4 we obtain on of the
main results about (\ref{eq8}). Furthermore, as an application we
find the complex-valued solutions $(f,g,h)$ of the functional
 equation\begin{equation}\label{EQ1}
f(xy)+\mu(y)g(\sigma(y)x)=h(x)h(y),\; x,y\in G\end{equation} on
groups and monoids in terms of multiplicative and additive
functions.
 \subsection{Notation and preliminary matters }
Throughout this paper $G$ denotes  a group and $S$ a semigroup. A
monoid is a semigroup $M$ with an identity element that we denote e.
The map $\sigma : S\longrightarrow S$ denotes an involutive
automorphism. That it is involutive means that $\sigma(\sigma(x)) =
x$ for all $x \in S$. The mapping $ \mu : S
\longrightarrow\mathbb{C}$ is a multiplicative function such that $
\mu(x\sigma(x))= 1$ for all $x \in S$.  If $\chi: S \longrightarrow
\mathbb{C}$ is a multiplicative  function such that $\chi\neq0$,
then $I_{\chi}=\{x\in S\; |\; \chi(x)=0\}$ is either empty or a
proper subset of $S$. Furthermore, $I_{\chi}$ is a two-sided ideal
in $S$ if not empty and $S\backslash I_{\chi}$ is a subsemigroup of
$S$. If $S$ is a topological space, then we let $\mathcal{C}(S)$
denote the algebra of continuous
functions from $S$ into $\mathbb{C}$.\\
For later use we need the following results. The next proposition
corresponds to Lemma 3.4 in [8].
\begin{prop}
Let $S$ be a semigroup, and suppose $f,g
:S\longrightarrow\mathbb{C}$ satisfy the sine addition law
\begin{equation}f(xy)=f(x)g(y)+f(y)g(x), \; x,y\in S\end{equation}
with $f\neq 0$. Then there exist multiplicative functions
$\chi_{1},\chi_{2}: G\longrightarrow \mathbb{C}$  such that
$$g=\frac{\chi_{1}+\chi_{2}}{2}.$$ Additionally we have the
following\\
i) If $\chi_{1}\neq \chi_{2}$ , then $f = c(\chi_{1}-\chi_{2})$ for some constant $c \in \mathbb{C}\setminus\{0\}$.\\
ii) If $\chi_{1}=\chi_{2}$, then letting $\chi:=\chi_{1}$ we have
$g=\chi$. If $S$ is a semigroup such that $S=\{xy\in :
x, y\in S\}$ (for instance a monoid), then $\chi\neq 0$.\\
If $S$ is a group, then there is an additive function $A: S\longrightarrow\mathbb{C}$, $A\neq 0$, such that $f=\chi A$. \\
If $S$ is a semigroup which is generated by its squares, then there
exists an additive function $A: S\setminus I_{\chi}\longrightarrow
\mathbb{C}$ for which
$$f(x)=\left\lbrace
\begin{array}{l}
\chi(x)A(x)\; \; for \;x\in S\setminus I_{\chi}\\
0    \; \; \;\; for \; x\in I_{\chi}.
\end{array}\right.$$
Furthermore, if $S$ is a topological group, or if $S$ is a
topological semigroup generated by its squares, and $f, g \in
\mathcal{C}(S)$, then $ \chi_{1}, \chi_{2}, \chi\in \mathcal{C}(S)$.
In the group case $A \in \mathcal{C}(S)$ and in the second case $A
\in \mathcal{C}(S\setminus I_{\chi})$.
\end{prop}
\section{$\mu$-sine subtraction law on a group and on a monoid}
In this section we deal with a new version of the sine subtraction
law
\begin{equation}\mu(y)k(x\sigma(y))=k(x)l(y)-k(y)l(x), \; x,y\in S,
\end{equation}
where $k, l$ are complex valued functions and $\mu$ is a
multiplicative function. The new feature is the introduction of the
function $\mu.$  We shall say that $k$ satisfies the $\mu$-sine
subtraction law with companion function $l$. If $S$ is a topological
semigroup and $k, l$ satisfy (2.1) such that $k\neq 0$ and $k$ is a
continuous function then $l$ is also a continuous function.
In the case where  $\mu=1$ and $G$ is a topological group, the functional equation (2.1) was solved in [8].\\
Here we focus exclusively on (2.1), and we include nothing about
other extensions of the cosine, sine addition and subtraction laws.
The next theorem is the analogue of Theorem 3.2 in [8].
\begin{thm}
Let $G$ be a group and let $\sigma: G\longrightarrow\mathbb{C}$ be a
involutive automorphism. Let $\mu: G\longrightarrow \mathbb{C}$ be a
multiplicative function such that $\mu(x\sigma(x))=1$ for all $x\in
G$. The solution $k, l: G\longrightarrow \mathbb{C}$ of the
$\mu$-sine subtraction law (2.1) with $k\neq 0$ are the following
pairs of functions, where $\chi: G\longrightarrow
\mathbb{C}\setminus\{0\}$
 denotes a character and $c_{1}\in \mathbb{C}$, $c_{2}\in \mathbb{C}\setminus\{0\}$ are constants.\\
i) If $\chi\neq \mu\;\chi\circ\sigma$, then
$$k=c_{2}\frac{\chi-\mu\;\chi\circ\sigma}{2}, \; l=\frac{\chi+\mu\;\chi\circ\sigma}{2}+c_{1}\frac{\chi-\mu\;\chi\circ\sigma}{2}.$$
ii) If $\chi=\mu\;\chi\circ\sigma$, then
$$k=\chi A,\;\; l=\chi(1+c_{1}A)$$
where $A: G\longrightarrow \mathbb{C}$ is an additive function such that $A\circ \sigma =-A\neq 0$.\\
Furthermore, if $G$ is a topological group and $k \in
\mathcal{C}(G)$, then $l, \chi, \mu\;\chi\circ\sigma, A \in
\mathcal{C}(G)$.
\end{thm}
\begin{proof}
By interchanging $x$ and $y$ in (2.1) we get that
$\mu(x)k(y\sigma(x))=-\mu(y)k(x\sigma(y))$ for all $x,y\in G$. By
setting $y=e$ we get that $k(x)=-\mu(x)k(\sigma(x))$ for all $x\in
G$. Using this,  equation (2.1) and the fact that
$\mu(x\sigma(x))=1$ we get for all $x,y\in G$ that
\begin{eqnarray*}
&&
k(x)[l(y)-\mu(y)l(\sigma(y))]-k(y)[l(x)-\mu(x)l(\sigma(x))]\\&&=k(x)l(y)-\mu(y)k(x)l(\sigma(y))-k(y)l(x)+\mu(x)k(y)l(\sigma(y))\\
&&=\mu(y)k(x\sigma(y)+\mu(xy)k(\sigma(x))l(\sigma(y))-\mu(xy)k(\sigma(y))l(\sigma(x))\\
&=&\mu(y)k(x\sigma(y))+\mu(xy)[k(\sigma(x))l(\sigma(y))-k(\sigma(y))l(\sigma(x))]\\&&=
\mu(y)k(x\sigma(y))+\mu(xy)\mu(\sigma(y))k(\sigma(x)y)\\&&=
\mu(y)k(x\sigma(y))+\mu(x)k(\sigma(x)y)\\&&=
mu(y)k(x\sigma(y))-\mu(x)\mu(\sigma(x)y)k(x\sigma(y))\\&&=
\mu(y)k(x\sigma(y))-\mu(y)k(x\sigma(y))\\&&= 0.
\end{eqnarray*}
So that we get for all $x,y\in G$  that
\begin{equation} k(x)[l(y)-\mu(y)l(\sigma(y))]=k(y)[l(x)-\mu(x)l(\sigma(x))].\end{equation} By using some ideas from [8]
we let $l^{+}(x)=\frac{l(x)+\mu(x)l(\sigma(x))}{2}$ and
$l^{-}(x)=\frac{l(x)-\mu(x)l(\sigma(x))}{2}$ for all $x\in G$. We
have $l=l^{+}+l^{-}$, $l^{+}(\sigma(x))=\mu(\sigma(x))l^{+}(x)$ and
$l^{-}(\sigma(x))=-\mu(\sigma(x))l^{-}(x)$ for all $x\in G$. From
(2.2) we have for all $x,y\in G$ that $k(x)l^{-}(y)=k(y)l^{-}(x)$.
Since $k\neq 0$  there exists an $x_{0}\in G$ such that
$k(x_{0})\neq 0$. Thus we have
$l^{-}(y)=\frac{l^{-}(x_{0})}{k(x_{0})}k(y)=ck(y)$ for all $y\in G$.
Then $l^{-}=ck$. On the other hand by replacing $l$ by $l^{+}+l^{-}$
in (2.1) and by using $k(x)l^{-}(y)=k(y)l^{-}(x)$ we get that
\begin{equation}
\mu(y)k(x\sigma(y))=k(x)l^{+}(y)-k(y)l^{+}(x), \; x,y\in G
\end{equation}
Replacing $y$ by $\sigma(y)$ in (2.3) we get that
\begin{equation}
k(xy)=k(x)l^{+}(y)+k(y)l^{+}(x), \; x,y\in G.
\end{equation}
According to Proposition 1.1 we have\\
either i) or ii) $k=c_{1}(\chi_{1}-\chi_{2})$ and
$l^{+}=\frac{\chi_{1}+\chi_{2}}{2}$ where $c_{1}\in
\mathbb{C}\setminus\{0\}$. Since $k(x)=-\mu(x)k(\sigma(x))$ and
$l^{+}(x)=\mu(x)l^{+}(\sigma(x))$ for all $x\in G$ we get that
$c_{1}(\chi_{1}-\chi_{2})=-c_{1}\mu(\chi_{1}\circ
\sigma-\chi_{2}\circ \sigma)$ and $\mu(\chi_{1}\circ
\sigma+\chi_{2}\circ \sigma)=\chi_{1}+\chi_{2}$. Then
$\chi_{2}=\mu\chi_{1}\circ \sigma$,
$l^{+}=\frac{\chi_{1}+\mu\chi_{1}\circ \sigma}{2}$ and
$k=c_{1}\frac{\chi_{1}-\mu\chi_{1}\circ \sigma}{2}$. Since $l^{-}
=ck=cc_{1}\frac{\chi_{1}-\mu\chi_{1}\circ
\sigma}{2}=c_{2}\frac{\chi_{1}-\mu\chi_{1}\circ \sigma}{2}$ where
$c_{2}\in \mathbb{C}\setminus\{0\}$.
By using the fact $l=l^{-}+l^{+}$ we get (i).\\
ii) we have $l^{+}=\chi$ and $k=\chi A$. Since
$l^{+}(\sigma(x))=\mu(\sigma(x))l^{+}(x)$ for all $x\in G$ we get
that $\chi=\mu\chi\circ \sigma$. From $k(x)=-\mu(x)k(\sigma(x))$ for
all $x\in G$ we get that $A=-A\circ \sigma$. By using the fact
$l=l^{-}+l^{+}=c\chi A+\chi=\chi(cA+1)$. This completes
 the proof.
\end{proof}
In the next proposition we extend Theorem 2.1 to monoids.
\begin{prop}
Let $M$ be a monoid and let $\sigma: M\longrightarrow M$ be a
involutive automorphism. Let $\mu: M\longrightarrow \mathbb{C}$ be a
multiplicative function such that $\mu(x\sigma(x))=1$ for all $x\in
M$. The solution $k, l: M\longrightarrow \mathbb{C}$ of the
$\mu$-sine
 subtraction law (2.1) with $k\neq 0$ are the following pairs of functions,
 where $\chi: M\longrightarrow \mathbb{C}\setminus\{0\}$ denotes a multiplicative
 function and $\chi(e)=1$ and $c_{1}\in \mathbb{C}$, $c_{2}\in \mathbb{C}\setminus\{0\}$ are constants :\\
i) If $\chi\neq \mu\;\chi\circ\sigma$, then
$$k=c_{1}\frac{\chi-\mu\;\chi\circ\sigma}{2}, \; l=\frac{\chi+\mu\;\chi\circ\sigma}{2}+c_{2}\frac{\chi-\mu\;\chi\circ\sigma}{2}.$$
ii) If $\chi=\mu\;\chi\circ\sigma$  and $S$ is generated by its
squares, then
$$k(x)=\chi(x) A(x),\;\; l(x)=\chi(x)(1+c_{1}A(x))\;  \forall x\in M\backslash I_{\chi} $$
$$l(x)=k(x)=0, \;\;\;\; \forall x\in  I_{\chi}$$
where $A: M\longrightarrow \mathbb{C}\setminus\{0\}$ is an additive function such that $A\circ \sigma =- A\neq  0$.\\
Furthermore, if $M$ is a topological monoid, and $k \in
\mathcal{C}(M)$, then $l, \chi, \mu\chi\circ \sigma \in
\mathcal{C}(M)$ and $A \in \mathcal{C}(M\backslash I_{\chi})$.
\end{prop}
\begin{proof}
By the same way in Proposition 3.6 in [8] and by using  Theorem 2.1
and the Proposition 1.1 we get the proof.
\end{proof}
\section{A variant of Wilson's functional equation on monoids} The
solutions of the functional equation (\ref{eq7}) on groups are
obtained in \cite{elqorachi}. More precisely, we have the following
theorem.\begin{thm} Let $G$ be a group, let $\sigma$ :
$G\longrightarrow G$ a homomorphism  such that $\sigma\circ\sigma =
I$, where $I$ denotes the identity map, and $\mu$ :
$G\longrightarrow \mathbb{C}$ be a  multiplicative function  such
that $\mu(x\sigma(x)) = 1$ for all $x\in G$. The solutions $f, g$ of
the functional equation (\ref{eq7}) are the following pairs of
functions, where $\chi$ : $G\longrightarrow \mathbb{C}$ denotes a
 multiplicative function and $c,\alpha\in \mathbb{C}^{\ast}$\\
(i) $f = 0$ and $g$ arbitrary.
\\(ii) $g = \frac{\chi+\mu \chi\circ \sigma}{ 2}$ and $f = \alpha g$.\\
 (iii) $g = \frac{\chi+\mu \chi\circ \sigma}{ 2}$ and $f = (c + \alpha/2 )\chi-(c -\alpha/2
 )\chi\circ \sigma$
with $(\mu-1)\chi = (\mu-1)\chi\circ \sigma$.\\
 (iv) $g = \chi$  and  $f=\chi(a+\alpha)$, where where $\chi=\mu\;\chi\circ\sigma$ and  $a$ is an additive map which satisfies
$a\circ\sigma+ a = 0$.\end{thm} We shall now extend this result to
monoids.
\begin{thm} Let $M$ be a monoid, let $\sigma$ :
$M\longrightarrow M$ a homomorphism  such that $\sigma\circ\sigma =
I$, where $I$ denotes the identity map, and $\mu$ :
$M\longrightarrow \mathbb{C}$ be a  multiplicative function   such
that $\mu(x\sigma(x)) = 1$ for all $x\in M$. The solutions $f, g$ of
the functional equation (\ref{eq7}) are the following pairs of
functions, where $\chi$ : $M\longrightarrow \mathbb{C}$ denotes a
 multiplicative function and $c,\alpha\in \mathbb{C}^{\ast}$\\
(i) $f = 0$ and $g$ arbitrary.
\\(ii) $g = \frac{\chi+\mu \chi\circ \sigma}{ 2}$ and $f = \alpha g$.\\
 (iii) $g = \frac{\chi+\mu \chi\circ \sigma}{ 2}$ and $f = (c + \alpha/2 )\chi-(c -\alpha/2
 )\chi\circ \sigma$
with $(\mu-1)\chi = (\mu-1)\chi\circ \sigma$.\\
 (iv) $g = \chi$, $f_e(xy)=f_e(x)\chi(y)+f_e(y)\chi(x)$
   for all $x,y\in S$ and with $f=f_e +f(e)\chi$. Furthermore, if  $M$ is a monoid which is generated by its squares, then $\chi=\mu\;\chi\circ\sigma$, there exists an additive function
 $a$ : $M\backslash I_{\chi}\longrightarrow \mathbb{C}$ for which $a\circ\sigma+ a =
 0$, $$f(x)=\left\lbrace
\begin{array}{l}
\chi(x)(a(x)+\alpha)\; \; for \;x\in M\setminus I_{\chi}\\
0    \; \; \;\; for \; x\in I_{\chi}
\end{array}\right.$$ Indeed, If $M$ is a topological group, or
$M$ is a topological monoid generated by its squares, $f,g,\mu\in
C(M)$, and $\sigma$ : $M\longrightarrow M$ is continuous, then
$\chi\in C(M)$. In the group case $a\in C(M)$ and in the second case
$a\in C(M\backslash I_{\chi})$.
   \end{thm}
   \begin{proof}Verifying  that the stated pairs of functions
   constitute solutions consists of simple computations. To see the
   converse, i.e., that any solution $f,g$ of (\ref{eq7}) is
   contained in one of the cases below,  we will use [8, Lemma 3.4] and [9, Theorem 3.1].
    All, except the last paragraphs of part (iv) and the continuity statements, is in Theorem 3.1
   in \cite{elqorachi}. Now, we assume that $M$ is a monoid generated by its squares. We use the notation used in the proof of Theorem 3.1
     in \cite{elqorachi}, in particular for the last paragraphs of part (iv) we have
   \begin{equation}\label{equation1}
   f_e(xy)=f_e(x)\chi(y)+f_e(y)\chi(x)\end{equation}
   for all $x,y\in S$ and with $f=f_e +f(e)\chi$. So, from [9, Lemma 3.4] we get $f(x)=0+f(e)\chi(x)=0+f(e)0=0$ if $x\in I_{\chi}$ and
   $f(x)=\chi(x)(a(x)+f(e))$ if $x\in S\backslash I_{\chi}$ and where $a$ is an additive function of $S\backslash I_{\chi}$.\\
    Now, we will verify that $\chi=\mu\;\chi\circ\sigma$ and
   $a\circ\sigma=-a$.  Since $f,g$ are solution of equation (\ref{eq7}) we have
 \begin{equation}\label{equation2}
   f(xy)+\mu(y)f(\sigma(y)x)=2f(x)\chi(y)\end{equation} for all $x,y\in M.$ By using the new expression of $f$ and the fact that $I_\chi$ is an
ideal, we get after an elementary computation that $f(xy)=f(yx)$ for
all $x,y\in M$. So, equation (\ref{equation2}) can be written as
follows
\begin{equation}\label{equation3}
   f(xy)+\mu(y)f(x\sigma(y))=2f(x)\chi(y),\;x,y\in
   M.\end{equation} By replacing $y$ by $\sigma(y)$ in
   (\ref{equation3}) and multiplying the result obtained by $\mu(y)$
   we get \begin{equation}\label{equation4}
   f(xy)+\mu(y)f(x\sigma(y))=2f(x)\mu(y)\chi(\sigma(y)),\;x,y\in
   M.\end{equation} Finally, by comparing  (\ref{equation3}),
   (\ref{equation4}) and using $f\neq 0$ we get
   $\chi(y)=\mu(y)\chi(\sigma(y))$ for all $y\in M.$\\
   By Substituting the expression of $f$  into (\ref{equation3}) and using $\chi(y)=\mu(y)\chi(\sigma(y))$ and $\mu(y\sigma(y))=1$ for all $y\in M$,
   we find after reduction that
   $\chi(x)\chi(y)[a(y)+a(\sigma(y))]=0$ for all $x,y\in M\backslash
   I_\chi$. Since $\chi\neq 0$ we get $a(y)+a(\sigma(y))=0$ for all $y\in M\backslash
   I_\chi$.\\ For the topological statement we use [13, Theorem
   3.18(d)]. This completes the proof.
\end{proof}
\section{Solutions of (\ref{eq8}) on groups and monoids}
In this section we solve the functional equation (\ref{eq8}) on
monoids. In the next proposition we show that if $(f,g,h)$ is a
complex-valued solution of equation (\ref{eq8}), then  $h$ satisfies
the $\mu$-sine subtraction law.
\begin{prop}
Let $M$ be a monoid, let $\sigma$ be a involutive automorphism on
$S$, let $\mu$ be a multiplicative function on $M$ such that
$\mu(x\sigma(x))=1$ for all $x\in M.$ Suppose that $f,g,h:
M\longrightarrow \mathbb{C}$ satisfy the functional equation
(\ref{eq8}). Suppose also that $g\neq 0$ and $h\neq 0$. Then\\
i) $h(\sigma(x))=-\mu(\sigma(x))h(x))$ for all $x\in M$.\\
ii) $h(xy)=h(yx)$ for all $x,y\in M$.\\
3i) $h$ satisfies the $\mu$-sine subtraction law (2.1).\\
4i) If $g(e)=0$, then $g=bh$ for some $b\in \mathbb{C}\setminus\{0\}$.\\
5i) If $g(e)\neq 0$, then $h$ satisfies the $\mu$-sine subtraction law with companion function $\frac{g}{g(e)}$.\\
Moreover, if $M$ is a topological monoid, and  $h\neq 0$ is
continuous, then the companion function  is also continuous.
\end{prop}
\begin{proof}
We follow  the path of the proof of Proposition 3.1 in [8].\\
By substituent $(x,yz), (\sigma(y),\sigma(z)x)$ and $(z,\sigma(xy))$
and $(z,\sigma(xy))$ in (\ref{eq8}) we obtain
 \begin{equation}
f(xyz)-\mu(yz)f(\sigma(yz)x)=g(x)h(yz),\end{equation}
\begin{equation}
f(\sigma(yz)x)-\mu(\sigma(z)x)f(z\sigma(xy))=g(\sigma(y))h(\sigma(z)x),\end{equation}
\begin{equation}
f(z\sigma(xy))-\mu(\sigma(xy))f(xyz)=g(z)h(\sigma(xy)).\end{equation}
By multiplying (4.1) by $\mu(\sigma(xy))$ we obtain that
\begin{equation}
\mu(\sigma(xy))f(xyz)-\mu(\sigma(x)z)f(\sigma(yz)x)=\mu(\sigma(xy))g(x)h(yz).\end{equation}
By adding (4.3) and (4.4) we obtain
\begin{equation}
f(z\sigma(xy))-\mu(\sigma(x)z)f(\sigma(yz)x)=g(z)h(\sigma(xy))+\mu(\sigma(xy))g(x)h(yz).\end{equation}
By multiplying (4.5) by $\mu(\sigma(z)x)$ we obtain
 \begin{equation}
\mu(\sigma(z)x)f(z\sigma(xy))-f(\sigma(yz)x)=\mu(\sigma(z)x)g(z)h(\sigma(xy))+\mu(\sigma(zy))g(x)h(yz).\end{equation}
By adding (4.6) and (4.2) we obtain
\begin{equation}0=g(\sigma(y))h(\sigma(z)x)+\mu(\sigma(z)x)g(z)h(\sigma(xy))+\mu(\sigma(zy))g(x)h(yz).\end{equation}
Setting $x_{0}$ such that $g(x_{0})\neq 0$ and the fact that
$\mu(x\sigma(x))=\mu(x)\mu(\sigma(x))=1$ for all $x\in M$, we get
that
\begin{equation}h(yz)=\mu(y)g(\sigma(y))l(z)+g(z)l_{1}(y), \; y,z\in M \end{equation}
where $l, l_{1}$ are complex valued functions on $M$. Using (4.8) in
(4.7) we obtain for all $x,y,z\in M$
\begin{equation}g(x)g(\sigma(y))\{l_{1}(\sigma(z))+\mu(\sigma(z))l(z)\} +
g(x)g(z)\{\mu(\sigma(z))l(\sigma(y))+\mu(\sigma(zy))l_{1}(y)\}
+\end{equation}
\begin{equation*}g(z)g(\sigma(y))\{\mu(\sigma(z))l(x)+\mu(\sigma(z)x)l_{1}(\sigma(x))\} =0.\end{equation*}
Putting $x=x_{0}$, $y=\sigma(x_{0})$, the equation (4.9) becomes
\begin{equation}l_{1}(\sigma(z))+\mu(\sigma(z))l(z)=c\mu(\sigma(z))g(z),\; z\in
M,\end{equation} where $c\in \mathbb{C}$ is a constant. By putting
(4.10) in (4.9) we obtain $3\mu(\sigma(z))cg(x)g(\sigma(y))g(z)=0$
for all $x,y,z\in S$. Since $g\neq 0$ and $\mu(\sigma(z))\neq 0$ it
follows that $c=0$ and then $l_{1}(\sigma(z))=-\mu(\sigma(z))l(z)$
for all $z\in S$. So that equation (4.8) becomes
\begin{equation}h(yz)=\mu(y)g(\sigma(y))l(z)-\mu(y)g(z)l(\sigma(y)),\; x,y,z\in S\end{equation}
Replacing $(y,z)$ by $(\sigma(z),\sigma(y))$ in (4.11) we obtain
\begin{equation}h(\sigma(zy))=\mu(\sigma(zy))(\mu(y)g(z)l(\sigma(y))-\mu(y)g(\sigma(y))l(z))=-\mu(\sigma(zy))h(yz).\end{equation}
From which we obtain by putting $y=e$ that
\begin{equation}h(\sigma(z))=-\mu(\sigma(z))h(z), \;z\in M.\end{equation}
From (4.12) and (4.13) we get that $h$ a central function i.e. $h(yz)=h(zy)$ for all $y,z\in M$.\\
Next we consider two cases :\\
First case : Suppose $g(e)=0$. Let $z=e$ and $x=x_{0}$ in (4.7) give
that
\begin{equation}h(y)=-c\mu(y)g(\sigma(y)), \; y\in M\end{equation}
for some $c\in \mathbb{C}\setminus\{0\}$. By replacing $y$ by
$\sigma(y)$ in (4.14) we obtain that
$h(\sigma(y))=-c\mu(\sigma(y))g(y)$ for all $y\in M$. By using
(4.14) we get that $g=\frac{1}{c}h=bh$ where $b=\frac{1}{c}$ and
that
\begin{equation}g(\sigma(z))=-\mu(\sigma(z))g(z), \;z\in
S.\end{equation} Using (4.15) in (4.11) and setting $m=-bl$ we get
\begin{eqnarray*}h(yz)&=&\mu(y)g(\sigma(y))l(z)-\mu(y)g(z)l(\sigma(y))\\
&=&-\mu(y)\mu(\sigma(y))g(y)l(z)-\mu(y)g(z)l(\sigma(y))\\
&=&-g(y)l(z)-\mu(y)g(z)l(\sigma(y))\\
&=&-bh(y)l(z)-\mu(y)bh(z)l(\sigma(y))\\
&=&h(y)m(z)+\mu(y)h(z)m(\sigma(y)).\end{eqnarray*} By replacing $z$
by $\sigma(z)$ we obtain
\begin{equation}h(y\sigma(z))=h(y)(m\circ \sigma)(z)+\mu(y)h(\sigma(z))m(\sigma(y)), \; y,z\in M.\end{equation}
By multiplying (4.16) by $\mu(z)$ and by setting $n(z)= (m\circ
\sigma)(z)\mu(z)$ we obtain the $\mu$-sine subtraction law  with the
companion function $n$
 \begin{equation}\mu(z)h(y\sigma(z))=h(y)n(z)-h(z)n(y).\end{equation} This ends the first case.\\
 Second case : Suppose $g(e)\neq 0$. Then we obtain from (4.7) with $x=e$ that
\begin{equation}h(yz)=[\mu(y)g(\sigma(y))h(z)+g(z)h(y)]/g(e).\end{equation}
Interchanging $y$ and $z$ in (4.18) we get
\begin{equation}h(zy)=[\mu(z)g(\sigma(z))h(y)+g(y)h(z)]/g(e).\end{equation}
By replacing $z$ by $\sigma(z)$ (4.19) and multiplying (4.19) by
$\mu(z)$   we get
\begin{equation}\mu(z)h(\sigma(z)y)=h(y)\frac{g}{g(e)}(z)-\frac{g}{g(e)}(y)h(z), \; y,z\in M.\end{equation}
Since $h$ is central it follows that $h$ satisfies the $\mu$-sine
subtraction law  with the companion function $\frac{g}{g(e)}$
\begin{equation}\mu(z)h(y\sigma(z))=h(y)\frac{g}{g(e)}(z)-\frac{g}{g(e)}(y)h(z), \; y,z\in M.\end{equation}
\end{proof}
In the next two theorems we obtain the solutions of equation (1.9)
by using our results  for the  $\mu$-sine subtraction law.  We will
follow the method used in [8].\\ Let $\mathcal{N}_{\mu}(\sigma,S)$
be the nullspace  given by
$$\mathcal{N}_{\mu}(\sigma,S)=\{\theta: S\longrightarrow\mathbb{C} \; :\theta(xy)-\mu(y)\theta(\sigma(y)x)=0, \;x,y\in G\}.$$
In the next theorem we consider the group case
\begin{thm}
Let $G$ be a group, let $\sigma$ be a involutive automorphism on
$G$, let $\mu:G \longrightarrow \mathbb{C}$ be a multiplicative
function such that $\mu(x\sigma(x))=1$ for all $x\in G.$  Suppose
that  $f, g, h: G \longrightarrow \mathbb{C}$ satisfy functional
equation (1.9). Suppose also that $g \neq 0$ and $h \neq 0$. Then
there exists a character $\chi$ of $G$, constants $c$, $c_1$,
$c_{2}\in \mathbb{C}$,
and a function $\theta\in \mathcal{N}_{\mu}(\sigma,S)$ such that one of the following holds\\
i) If $\chi\neq \mu\;\chi\circ\sigma$, then
$$h=c_{1}\frac{\chi-\mu\;\chi\circ\sigma}{2}, \; g=\frac{\chi+\mu\;\chi\circ\sigma}{2}+c_{2}\frac{\chi-\mu\;\chi\circ\sigma}{2},$$
$$f=\theta+\frac{c_{1}}{2}[c\frac{\chi-\mu\;\chi\circ\sigma}{2}+c_{2}\frac{\chi+\mu\;\chi\circ\sigma}{2}]$$
ii) If $\chi=\mu\;\chi\circ\sigma$, then
$$h=\chi A,\;\; g=\chi(c+c_{2}A), f=\theta+\chi A(\frac{c}{2}+\frac{c_{2}}{4}A).$$
where $A: G\longrightarrow \mathbb{C}\setminus\{0\}$ is an additive function such that $A\circ \sigma =-A\neq 0$.\\
Conversely, the formulas of (i) and (ii) define solutions of (1.9).\\
Moreover, if $G$ is a topological group, and and $f, g, h \in
\mathcal{C}(G)$, then $\chi, \mu\;\chi\circ\sigma, A, \theta\in
\mathcal{C}(G)$, while $A \in  \mathcal{C}(G)$.
\end{thm}
The proof of the theorem 4.2 will be integrated into that of theorem
4.3 in which we consider the monoid case
\begin{thm}
Let $M$ be a monoid which is generated by its squares, let $\sigma$
be an involutive automorphism on $M$, let $\mu:M \longrightarrow
\mathbb{C}$ be a multiplicative function such that
$\mu(x\sigma(x))=1$ for all $x\in M.$ Suppose that $f, g, h: M
\longrightarrow \mathbb{C}$ satisfy functional equation (1.9).
Suppose also that $g \neq 0$ and $h \neq 0$. Then there exists a
multiplicative function $\chi: M\longrightarrow \mathbb{C}$
$\chi\neq 0$, constants $c$, $c_{1}$, $c_{2}\in \mathbb{C}$,
and a function $\theta\in \mathcal{N}_{\mu}(\sigma,S)$ such that one of the following holds\\
i) If $\chi\neq \mu\;\chi\circ\sigma$, then
$$h=c_{1}\frac{\chi-\mu\;\chi\circ\sigma}{2}, \; g=\frac{\chi+\mu\;\chi\circ\sigma}{2}+c_{2}\frac{\chi-\mu\;\chi\circ\sigma}{2},$$
$$f=\theta+\frac{c_{1}}{2}[c\frac{\chi-\mu\;\chi\circ\sigma}{2}+c_{2}\frac{\chi+\mu\;\chi\circ\sigma}{2}]$$
ii) If $\chi=\mu\;\chi\circ\sigma$, then $h(x)=g(x)=0$ and
$f(x)=\theta(x)$ for $x\in I_{\chi}$, and
$$h(x)=\chi(x)A(x),\;\; g(x)=\chi(x)(c+c_{2}A(x)), f(x)=\theta(x)+\chi(x) A(x)(\frac{c}{2}+\frac{c_{2}}{4}A(x))$$ for $x\in M\setminus I_{\chi}$
where $A: M\setminus I_{\chi}\longrightarrow \mathbb{C}\setminus\{0\}$ is an additive function such that $A\circ \sigma =-A\neq 0$.\\
Furthermore, if $M$ is a topological monoid  and $k \in C(M)$, then
$l$, $\chi$, $\mu\chi\circ \sigma$, $A\in \mathcal{C}(M)$.
Conversely, the formulas of (i) and (ii) define solutions of (1.9).\\
Moreover, if $M$ is a topological monoid, and $f, g, h \in
\mathcal{C}(M)$, then $\chi, \mu\;\chi\circ\sigma, A, \theta\in
\mathcal{C}(M)$, while $A\in \mathcal{C}(M\setminus I_{\chi})$.
\end{thm}
\begin{proof}
According to Proposition 4.1 we have two following cases: \\
First case : Suppose that $g(e)=0$, then $h$ satisfies the
$\mu$-sine subtraction law and $g=bh$ where $b\in
\mathbb{C}\setminus\{0\}$. According to Theorem 2.1 and Proposition
2.2 we get (for $M$ is a group or a monoid) that if $\chi\neq
\mu\chi\circ \sigma$, then
$h=c_{1}\frac{\chi-\mu\;\chi\circ\sigma}{2}$ where $c_{1}\in
\mathbb{C}$. Since $g=bh$, then
$g=c_{2}\frac{\chi-\mu\;\chi\circ\sigma}{2}$ where $c_{2}\in
\mathbb{C}$. Subsisting $g$ and $h$ in (3.1) we get for all $x,y\in
M$
\begin{eqnarray*}f(xy)-\mu(y)f(\sigma(y)x)&=&g(x)h(y)\\
&=&\frac{c_{1}c_{2}}{4}[\chi(x)-\mu(x)\chi(\sigma(x))][\chi(y)-\mu(y)\chi(\sigma(y))]\\&=&
\frac{c_{1}c_{2}}{4}[\chi(xy)-\mu(y)\chi(\sigma(y)x)-\mu(x)\chi(\sigma(\sigma(y)x))\\&&+\mu(xy)\chi(\sigma(xy))].\end{eqnarray*}
Let $\theta=f-\frac{c_{1}c_{2}}{4}(\chi+\mu\;\chi\circ\sigma)$ we
have $\theta\in \mathcal{N}_{\mu}(\sigma,S)$ and
$f=\theta+\frac{c_{1}}{2}[c\frac{\chi-\mu\;\chi\circ\sigma}{2}+c_{2}\frac{\chi+\mu\;\chi\circ\sigma}{2}]$ with $c=0$.\\
Now, if $\chi= \mu\chi\circ \sigma$.\\
When $M$ is a group then we get from Theorem 2.1 that $h=\chi A$
where $A$ is an additive function such that $A\circ \sigma=-A\neq
0$. Since $g=bh$, then $g=b\chi A=c_{2}\chi A$. Subsisting $g$ and
$h$ in (3.1) we get
\begin{eqnarray*} f(xy)-\mu(y)f(\sigma(y)x)&=&g(x)h(y)c_{2}\chi(x)A(x)\chi(y)A(y)\\
&=&\frac{c_{2}}{4}[\chi(y)A(xy)^{2}-\chi(\sigma(y))A(\sigma(y)x)^{2}].\end{eqnarray*}
Let $\theta=f-c_{2}\frac{\chi A^{2}}{4}$. Then $\theta\in \mathcal{N}_{\mu}(\sigma,S)$.\\
When $M$ is a monoid, by Proposition 2.2 and and by the same way as
in [6] we have $\theta=\theta_{1}\cup \theta_{2}$ where
where $\theta_{1}(x)= f(x)-c_{2}\frac{\chi(x) A^{2}(x)}{4}$ on $M\setminus I_{\chi}$ and $\theta_{2}(x)=f(x)$ on $I_{\chi}$.\\
Second case : Suppose $g(e)\neq 0$. We have $h$ and $\frac{g}{g(e)}$
play the role of $k$ and $l$ respectively in Theorem 2.1 or in
Proposition 2.2. If $\chi\neq \mu\;\chi\circ\sigma$, then
$h=c_{1}\frac{\chi-\mu\;\chi\circ\sigma}{2}$ where $c_{1}\in
\mathbb{C}$ and
$g=c\frac{\chi+\mu\;\chi\circ\sigma}{2}+c_{2}\frac{\chi-\mu\;\chi\circ\sigma}{2}$.
By the same way as in [9] we get that $\theta=f-c_{1}\frac{[(c+c_{2})\chi-(c-c_{2})\chi\circ\sigma]}{4}\in\mathcal{N}_{\mu}(\sigma,S)$.\\
Finally, if $g(e)\neq 0$ and $\chi=\mu\;\chi\circ\sigma$ we get the
remainder by the same way as in [9].
\end{proof}\section{Applications: Solutions of equation (\ref{EQ1})
on groups and monoids}In this section, we use the results obtained
in the previous paragraph to solve the functional equation
(\ref{EQ1}) on groups and monoids. We proceed as follows to reduce
the equation to the functional equation (\ref{eq7}) and  (\ref{eq8})
so that we can apply Theorem 3.1, Theorem 3.2,  Theorem 4.2. and
Theorem 4.3.
\begin{thm} Let $G$ be a group, and $\sigma$ an homomorphism
involutive of  $G$. Let $\mu:G \longrightarrow \mathbb{C}$ be a
multiplicative function such that $\mu(x\sigma(x))=1$ for all $x\in
G$. Suppose that the functions  $f, g, h: G \longrightarrow
\mathbb{C}$ satisfy the functional equation (\ref{EQ1}). Suppose
also that $f+g\neq0.$ Then there exists a character $\chi$ of $G$,
constants $\alpha\in \mathbb{C}^{\ast}$, $c_1$, $c_{2}\in
\mathbb{C}$, and a function $\theta\in \mathcal{N}_{\mu}(\sigma,S)$
such that one of the following holds\\(a) If $\chi\neq
\mu\;\chi\circ\sigma$, then
$f=\frac{1}{2}[(1+\frac{c_1c_2}{2})\frac{\chi+\mu\;\chi\circ\sigma}{2}+2c_2\frac{\chi-\mu\;\chi\circ\sigma}{2}+\theta]$;
    $g=\frac{1}{2}[(1-\frac{c_1c_2}{2})\frac{\chi+\mu\;\chi\circ\sigma}{2}-\theta]$ and
    $h=\frac{1}{\alpha}[\frac{\chi+\mu\;\chi\circ\sigma}{2}+c_2\frac{\chi-\mu\;\chi\circ\sigma}{2}]$.\\(b) If $\chi= \mu\;\chi\circ\sigma$ then  $f=\frac{c_2\chi(2+A)+\theta+\chi
   A c_2(1+\frac{A}{4})}{2}$; $g=\frac{c_2\chi(2+A)-\theta-\chi
   A c_2(1+\frac{A}{4})}{2}$ and $h=\frac{1}{\alpha}c_2\chi(2+A).$
\end{thm}
\begin{proof} Let $f,g,h$ : $G\longrightarrow \mathbb{C}$
satisfy the functional equation (\ref{EQ1}). The case $g=-f$ was
treated in Theorem 3.1 and Theorem 3.2. From now, on we assume that
$f+g\neq 0$. Let  $h_o:=\frac{h-\mu h\circ \sigma}{2}$ respectively
$h_e:=\frac{h+\mu h\circ \sigma}{2}$ denote the odd, respectively
even, part of $h$ with respect to $\mu$ and $\sigma$.\\
Setting $x=e$ in (\ref{EQ1}) gives us
\begin{equation}\label{equation51}
f(y)+\mu(y)g(\sigma(y))=h(e)h(y)\end{equation} for all $y\in G.$
Taking $y=e$ in (\ref{EQ1}) and using $\mu(e)=1$ we find
\begin{equation}\label{equation52}
f(x)+g(x)=h(e)h(x)\end{equation} for all $x\in G.$ So, by comparing
(\ref{equation51}) with (\ref{equation52}) we get
\begin{equation}\label{equation53}
g(x)=\mu(x)g(\sigma(x)),\;x\in G.\end{equation}
 We note that
 $(f+g)(xy)+\mu(y)(f+g)(\sigma(y)x)=f(xy)+\mu(y)g(\sigma(y)x)+g(xy)+\mu(y)f(\sigma(y)x)=h(x)h(y)+g(xy)+\mu(y)f(\sigma(y)x)$.
 By using (\ref{equation53}) we have
 $g(xy)=\mu(xy)g(\sigma(x)\sigma(y))$, then we get
 $g(xy)+\mu(y)f(\sigma(y)x)=\mu(y)f(\sigma(y)x)+\mu(xy)g(\sigma(x)\sigma(y))=\mu(y)[f(\sigma(y)x)+\mu(x)g(\sigma(x)\sigma(y))]=\mu(y)h(x)h(\sigma(y))$.
 Which implies that
 \begin{equation}\label{equation540}
 (f+g)(xy)+\mu(y)(f+g)(\sigma(y)x)=2h(x)h_e(y)\end{equation} for all
 $x,y\in G.$ From (\ref{equation52}) and the assumption that $f+g\neq
 0$ we get $h(e)\neq 0$. So, equation (\ref{equation540}) can be
 written as follows \begin{equation}\label{equation541}
 (f+g)(xy)+\mu(y)(f+g)(\sigma(y)x)=2(f+g)(x)\frac{h_e(y)}{h(e)}\end{equation} for all
 $x,y\in G.$ \\On the
 other hand by using similar computation used above, we obtain
\begin{equation}\label{equation542}
 (f-g)(xy)-\mu(y)(f-g)(\sigma(y)x)=2h(x){h_o(y)}=(f+g)(x)\frac{2h_o(y)}{h(e)}\end{equation} for all
 $x,y\in G$.  We can now  apply Theorem 3.1, Theorem 3.2, Theorem 4.2 and Theorem 4.3.\\
 If $h_o=0$, then $f-g\in \mathcal{N}(\sigma,G)$, so there exists $\theta\in \mathcal{N}_{\mu}(\sigma,G)$
 such that $f-g=\theta.$ Since $f,g$ satisfy (\ref{equation541}) then from Theorem 3.1
 we get the only possibility $f+g=\alpha^{2}\frac{\chi+\mu\;\chi\circ\sigma}{2}$ and $h=\alpha\frac{\chi+\mu\;\chi\circ\sigma}{2}$ for some character $\chi$ :
  $G\longrightarrow \mathbb{C}$ and a constant $\alpha\in \mathbb{C}$ and  we deduce that $f=\frac{1}{2}[\theta+\alpha^{2}(\frac{\chi+\mu\;\chi\circ\sigma}{2})]$;
  $g=\frac{1}{2}[-\theta+\alpha^{2}(\frac{\chi+\mu\;\chi\circ\sigma}{2})]$. We deal with case (i).\\
  So during the rest of the proof we will assume that $h_o\neq0.$ the function $f+g,$ $h_o$ are solution of equation (\ref{equation542})
  with $f+g\neq 0$ and $h_o\neq 0$, so we know from
  Theorem 4.2 that there are only the following $2$ possibilities:\\
  (i)
  $f-g=\theta+\frac{c_1}{2}[c\frac{\chi-\mu\;\chi\circ\sigma}{2}+c_2\frac{\chi+\mu\;\chi\circ\sigma}{2}]$;
  $f+g=\frac{\chi+\mu\;\chi\circ\sigma}{2}+c_2\frac{\chi-\mu\;\chi\circ\sigma}{2}$
  for some character $\chi$ on $G$ such that $\chi\neq\mu\;\chi\circ\sigma$, $\theta\in \mathcal{N}_{\mu}(\sigma,G)$ and constants $c,c_1,c_2\in
  \mathbb{C}$. So, we have
  $g=\frac{1}{2}[(1-\frac{c_1c_2}{2})\frac{\chi+\mu\;\chi\circ\sigma}{2}+(c_2-\frac{cc_1}{2})\frac{\chi-\mu\;\chi\circ\sigma}{2}-\theta]$;
   $f=\frac{1}{2}[(1+\frac{c_1c_2}{2})\frac{\chi+\mu\;\chi\circ\sigma}{2}+(c_2+\frac{cc_1}{2})\frac{\chi-\mu\;\chi\circ\sigma}{2}+\theta]$.
   Since $g=\mu g\circ\sigma$, then we have $c_2=\frac{cc_1}{2}$ and $f$, $g$ are as
   follows:
    $f=\frac{1}{2}[(1+\frac{c_1c_2}{2})\frac{\chi+\mu\;\chi\circ\sigma}{2}+2c_2\frac{\chi-\mu\;\chi\circ\sigma}{2}+\theta]$;
    $g=\frac{1}{2}[(1-\frac{c_1c_2}{2})\frac{\chi+\mu\;\chi\circ\sigma}{2}-\theta]$.  We
   deal with case (a).\\
   (ii) $f-g=\theta+\chi A(\frac{c}{2}+\frac{c_2}{4}A)$;
   $f+g=\chi(c+c_2A)$, where $\chi$ is a character on $G$ such that
   $\chi=\mu\;\chi\circ\sigma$, $A$ is an additive map on $G$ such
   that $A\circ\sigma=-A$, $\theta\in \mathcal{N}_{\mu}(\sigma,G)$ and $c,c_2\in
   \mathbb{C}$. So, we get $g=\frac{\chi(c+c_2A)-\theta-\chi
   A(\frac{c}{2}+\frac{c}{4})A}{2}$. Since $g=\mu g\circ\sigma$ so,
   we have $c=2c_2$. Consequently we have $g=\frac{c_2\chi(2+A)-\theta-\chi
   A c_2(1+\frac{A}{4})}{2}$; $f=\frac{c_2\chi(2+A)+\theta+\chi
   A c_2(1+\frac{A}{4})}{2}$, where $A$ : $G \longrightarrow \mathbb{C}$ is an additive function such that
$A\circ\sigma(x)=-A(x)$ for all $x\in G$. We deal with case (b).
   \end{proof}
\begin{thm} Let $M$ be a monoid which is generated by its squares, let $\sigma$ an
 involutive automorphism on  $M$. Let $\mu:M\longrightarrow \mathbb{C}$ be a
multiplicative function such that $\mu(x\sigma(x))=1$ for all $x\in
M$. Suppose that $f, g, h: M \longrightarrow \mathbb{C}$ satisfy the
functional equation (\ref{EQ1}). Suppose also that $f+g\neq0.$ Then
there exists a character $\chi$ of $M$, constants $\alpha\in
\mathbb{C}^{\ast}$, $c_1$, $c_{2}\in \mathbb{C}$, and a function
$\theta\in \mathcal{N}_{\mu}(\sigma,M)$ such that one of the
following holds\\(a) If $\chi\neq \mu\;\chi\circ\sigma$, then
$f=\frac{1}{2}[(1+\frac{c_1c_2}{2})\frac{\chi+\mu\;\chi\circ\sigma}{2}+2c_2\frac{\chi-\mu\;\chi\circ\sigma}{2}+\theta]$;
    $g=\frac{1}{2}[(1-\frac{c_1c_2}{2})\frac{\chi+\mu\;\chi\circ\sigma}{2}-\theta]$ and
    $h=\frac{1}{\alpha}[\frac{\chi+\mu\;\chi\circ\sigma}{2}+c_2\frac{\chi-\mu\;\chi\circ\sigma}{2}]$.
    \\(b) If $\chi= \mu\;\chi\circ\sigma$ then $f(x)=\frac{\theta(x)}{2}$, $g(x)=\frac{-\theta(x)}{2}$ and $h(x)=0$ for all $x\in I_\chi$;
    $f(x)=\frac{c_2\chi(2+A(x))+\theta(x)+\chi(x)
   A c_2(1+\frac{A(x)}{4})}{2}$;
   $g(x)=\frac{c_2\chi(2+A(x))-\theta(x)-\chi(x)
   A c_2(1+\frac{A(x)}{4})}{2}$ and
   $h(x)=\frac{1}{\alpha}c_2\chi(2+A(x))$ for all $x\in M\backslash I_\chi$ and where $A$ : $M\backslash I_\chi \longrightarrow \mathbb{C}$   is an additive function such that
$A\circ\sigma(x)=-A(x)$ for all $x\in M\backslash I_\chi$.
\end{thm}

\vspace{1cm}
Belaid Bouikhalene\\ Departement of Mathematics and Informatics\\
Polydisciplinary Faculty, Sultan Moulay Slimane university, Beni Mellal, Morocco.\\
E-mail : bbouikhalene@yahoo.fr.\\\\
 Elhoucien Elqorachi, \\Department of Mathematics,
\\Faculty of Sciences, Ibn Zohr University, Agadir,
Morocco,\\
E-mail: elqorachi@hotmail.com

\end{document}